\documentclass[a4paper,12pt]{amsart}
\usepackage{amssymb,amsmath,amsfonts, amsthm}
\usepackage{hyperref}
\usepackage{dsfont}
\usepackage{upgreek}
\usepackage{mathrsfs}
\usepackage{mathabx,yfonts}

\newtheorem{thm}{Theorem}

\newtheorem{lem}{Lemma}

\numberwithin{equation}{section}

\DeclareSymbolFont{bbold}{U}{bbold}{m}{n}
\DeclareSymbolFontAlphabet{\mathbbold}{bbold}
\renewcommand{\mod}[1]{\ \left(\textnormal{mod}\ #1\right)}
 
\renewcommand{\P}{\mathbb{P}}
\newcommand{\Q}{\mathbb{Q}}
\newcommand{\R}{\mathbb{R}}
\newcommand{\B}{\mathcal{B}}
\newcommand{\Z}{\mathbb{Z}}
\newcommand{\Zp}{\mathbb{Z}_{prim}}
\renewcommand{\l}{\left}
\renewcommand{\r}{\right}
\renewcommand{\b}{\mathbf}
\renewcommand{\c}{\mathcal}
\renewcommand{\epsilon}{\varepsilon}
\renewcommand{\gcd}{\textrm{gcd}}
\renewcommand{\u}{\widetilde{1}}
\renewcommand{\div}{\widetilde{\tau}}
\renewcommand{\leq}{\leqslant}
\renewcommand{\geq}{\geqslant}
\renewcommand{\#}{\sharp}
\renewcommand{\gg}{\ggg}
\renewcommand{\ll}{\lll}
\newcommand{\rr}{\widetilde{r}}
\newcommand{\pp}{\widetilde{\pi}}

\DeclareMathOperator*{\Osum}{\sum{}^*}

\begin{document}

\title{Rational Points on the Fermat Cubic Surface}

\author{Efthymios Sofos}
\address{School of Mathematics, University of Bristol, Bristol, BS8 1TW, United Kingdom}
\email{efthymios.sofos@bristol.ac.uk}

\begin{abstract}
We prove a lower bound that
agrees with Manin's prediction
for the number of rational points of bounded height
on the Fermat cubic surface. As an 
application we provide a simple 
counterexample to Manin's conjecture over $\Q.$
\end{abstract}


\maketitle

\section{Introduction}
\label{1s}
The subject of representing integers as a sum of two cubes
has a rich history in number theory. 
Examples such as
$91\!=\!6^3\!-\!5^3\!=\!4^3\!+\!3^3$
and
$1729\!=\!1^3\!+\!12^3\!=\!9^3\!+\!10^3$
make one wonder how often 
integers with at least two essentially distinct representations occur
and Euler had in fact showed that 
there are arbitrarily large such integers. 
Let $F\!=\!0$
denote the Fermat cubic surface,
where $F$ is given by
\[F:= x_0^3+x_1^3+x_2^3+x_3^3,\]
and notice that 
essentially non--distinct representations
give rise to elements $\b{x} \in \Z^4_{prim}$
on the lines 
\[x_0\!+\!x_1\!=\!x_2\!+\!x_3\!=\!0, \ \ \ \
x_0\!+\!x_2\!=\!x_1\!+\!x_3\!=\!0, \ \ \ \
x_0\!+\!x_3\!=\!x_1\!+\!x_2\!=\!0
\]
of the surface,
where $\Z_{prim}^4$ denotes 
the set of integer vectors $\b{x}$ such that
$\gcd(x_0,\ldots,x_3)=1.$
We are interested in estimating
the growth of 
the counting function 
\[
N\!\left(B\right)=
\#\l\{\b{x} \in \Zp^4,
F\!\left(\b{x}\right)\!=\!0, \b{x} \ \text{outside lines}, |\b{x}|\leq B\r\},
\]
where $|\cdot|$ is the usual supremum norm.

Hooley~\cite{hooley0}, building upon work of
Erd\"{o}s~\cite{erdos},
showed that
for almost all integers which are the sum of two cubes the representation is unique,
by proving that $N\!\l(B\r)\!=o\!\l(B^2\r).$
In a subsequent revisit of the subject~\cite[Th.3 \& Th.4]{hooley} 
he used Deligne's estimates along with sieve arguments to establish the stronger estimates
\begin{equation}
\label{hool}
B \!\ll \! N\!\left(B\right) \! \ll \! B^{\frac{5}{3}+\epsilon}
\end{equation}
for all $\epsilon\!>\!0$ and $B\!\geq \!6.$
These estimates
raised the question
of evaluating the true asymptotic order of $N\!\left(B\right).$
A conjectural answer was provided by 
a special case of a very general conjecture due to
Manin~\cite{fmt}.
Whenever  $f\!=\!0$ is a general smooth cubic surface
with a rational point, the conjecture states that
\[
N_f\!\l(B\r)
\sim
c_f B \left(\log B\right)^{\rho_f-1},
\]
as $B\to \infty,$
where $N_f\!\l(B\r)$ is defined similarly to $N\!\l(B\r)$ by excluding the $27$ lines contained in $f\!=\!0,$
$\rho_f \in \![1,7]$ is the rank of the Picard group of the surface
and $c_f$
is a positive constant.
It is worth noticing that 
although
the conjecture has been established for a large class 
of varieties, see for example~\cite{gauss},
it has never been
established for a single
smooth cubic surface.

In the case of the Fermat cubic surface the conjecture predicts the asymptotic behaviour
\[N\!\l(B\r)
\sim
c B \left(\log B\right)^{3},
\]
for some $c\!>\!0,$
since it is known that $\rho_{F}\!=\!4$ (see \S 8.3.1 in~\cite{quant}).
Wooley~\cite{wooley}
gave an elementary proof of the upper bound in~\eqref{hool}
and Heath--Brown~\cite{43},
building on this work,
made the improvement
\[N\!\left(B\right)\ll_{\epsilon}B^{\frac{4}{3}+\epsilon},
\]
for any $\epsilon\!>\!0,$ 
an estimate which actually applies to any smooth cubic surface with $3$ rational
coplanar lines.
Regarding lower bounds,
Slater and Swinnerton--Dyer~\cite{slat}
used a secant and tangent process
to obtain 
\begin{equation}
\label{slater}
N_f\left(B\right)\gg_f B \left(\log B\right)^{\rho_f-1},
\end{equation}
whenever the smooth cubic surface
$f=0$ contains two skew lines defined over the rationals.
The result does
not however
cover the case of the Fermat cubic surface since its
only skew lines are defined over
$\mathbb{Q}\left(\sqrt{-3}\right).$
Our main goal is to fill this gap and to improve optimally the lower bound in~\eqref{hool}.
\begin{thm}
\label{bonn}
We have the estimate
\[N\!\left(B\right)\gg B \left(\log B\right)^{3},
\] 
for all $B\geq 6.$
\end{thm}
Our key ingredient here is earlier work of the author~\cite{sofos}.
It allows us to cover the surface with a family of conics and count rational points on each conic individually.
This approach leads to a sum of Hardy--Littlewood densities. We will show that
this sum behaves as a divisor sum involving binary forms in \S\ref{s3}.
This approach can be used 
to prove~\eqref{slater} for other
smooth cubic surfaces with a rational line,
a topic that we intend to return to in the future.
In fact,
we note that 
our method, when fully extracted, is capable of proving
\[N_f\!\l(B\r)\geq
\delta 
c_f 
B 
\l(\log B\r)^{\rho_f-1},
\]
for all large enough $B,$
where $\delta\!>\!0$ is a small absolute constant
and $c_{\!f}$ is the Peyre constant~\cite{peyre}.

We will use Theorem~\ref{bonn} in~\S\ref{s6}
in order to provide a simple counterexample over $\Q$
to the full 
version
of Manin's conjecture which is formulated for
general Fano varieties.
Let 
$k$ 
be a number field,
$Z$ be a Fano variety
over $k$
and $H$ be an anticanonical height function
on $Z.$
Suppose that
$Z\!\l(k\r)$ is Zariski dense
and define
the counting function 
\begin{equation}
\label{nam0}
N_{k}\l(U,B\r):
=\#\l\{x \in U\l(k\r), H\l(x\r)\leq B\r\}
\end{equation} 
for $B\geq 1,$
where $U$ is a Zariski open subset of $Z.$
A lack of the subscript $k$ will indicate that 
the counting is performed over the rationals.

Manin's conjecture~\cite{fmt} describes the growth rate
of $N_{k}\l(U,B\r)$ 
in terms of geometric invariants 
related to $Z.$ It
states that
there exists a Zariski open $U\subset Z$
and a constant $c=c\l(U,k, H\r)>0$
such that 
\begin{equation}
\label{nam}
N_k\l(U,B\r)\sim c B \l(\log B\r)^{\rho-1},
\end{equation}
as $B \to \infty,$
where 
$\rho$ is the rank of the Picard group of $Z.$

The conjecture does not hold in full generality, the first counterexample
having been
provided by Batyrev and Tschinkel~\cite{bat}. They
consider the biprojective hypersurface Fano cubic bundle 
$Y \!\subset 
\!
\mathbb{P}_{\!k}^3
\! \times \! 
\mathbb{P}_{\!k}^3$
given by 
\begin{equation}
\label{btt}
\sum_{i=0}^3 x_iy_i^3=0.
\end{equation}
It is shown that
if $k$ contains
a cube root of unity
and $U$ is any nonempty Zariski open subset of $Y$
we have
\[N_k\l(U,B\r) \gg B \l(\log B\r)^{3}
.\]
This estimate
disproves~\eqref{nam} 
since the Picard group of $Y$ is isomorphic to 
$\mathbb{Z}^2,$
as shown in~\cite[Prop.1.3]{bat}.

This result however leaves a counterexample 
over $\Q$ to be desired. This was achieved
by Loughran~\cite{dan}, where
Weil restriction was used to 
provide implicit
counterexamples
over any number field $k$
and of arbitrarily large dimension.
Further counterexamples 
related to the Peyre constant 
and the power of the logarithm
in~\eqref{nam}
are provided
by Browning and Loughran~\cite{timdan} and
Le Rudulier~\cite{cecile} respectively.

Our aim is to extend
the Batyrev--Tschinkel counterexample
over $\Q.$
\begin{thm}
\label{batt}
For any nonempty Zariski open $U\subset Y,$
where $Y$ is the
biprojective hypersurface 
given by~\eqref{btt},
we have
\[N\l(U,B\r)\gg B \l(\log B\r)^3,
\]
where the counting is performed over $\Q.$
\end{thm}
This estimate contradicts Manin's conjecture~\eqref{nam}
due to the incompatibility of logarithmic exponents.
Although we will not give details,
our method
is capable of proving that
Manin's conjecture
is not valid for~\eqref{btt} over any number field.

\textbf{Acknowledgements}: I would like to express my gratitude to Dr. D. Loughran for suggesting the 
application to counterexamples
and for various helpful discussions, 
to Dr. D. Schindler for comments on an earlier draft of this paper,
and to Prof. T. Browning, who suggested
the project to me and who has been most generous with his help
and advice at all stages of the work.
While working on this paper the author was supported by \texttt{EPSRC} grant
\texttt{EP/H005188/1}.

\textbf{Notation}:
For any functions
$f,g :[1,\infty) \to \mathbb{C},$
the equivalent notations
$f\!\l(x\r)\!=\!O_{\mathcal{S}}\l(g\!\l(x\r)\r),$
and
$f\!\l(x\r) \! \ll_{\mathcal{S}} \! g\!\l(x\r),$
will be used to denote the 
existence of a positive constant
$\lambda,$ which depends at most 
on the set of parameters $\mathcal{S}$
such that
for any
$x\geq 1$
we have
$\l|f\l(x\r)\r|\leq \lambda \l|g\l(x\r)\r|.$
Throughout sections
\S\ref{s3}--\S\ref{s5}
the absence of 
a subscript $\mathcal{S}$
will indicate 
that the implied constant
is absolute. As usually,
we denote
the Euler, M\"{o}bius and the divisor function
by $\phi,\mu$ and $\tau$ respectively, and we let 
$\omega\!\l(n\r)$ be the number of distinct prime divisors of $n.$
We shall make frequent use of the 
familiar estimate
\begin{equation}
 \label{eq:1}
  \tau\l(n\r) \ll_{\epsilon}  n^{\epsilon},
\end{equation} 
valid for any $\epsilon>0.$

\section{Covering by conics}
\label{s2}
The identity
$x^3+y^3=
\frac{1}{4}
\l(x+y\r)
\l(\l(x+y\r)^2
+3
\l(x-y\r)^2
\r)$
reveals that
the Fermat surface
$F\!=\!0$
is equivalent
over $\Q$
to 
\[X : \hspace{0.3cm}
x_0\l(x_0^2+3x_1^2\r)
=
x_2\l(x_2^2+3 x_3^2\r),\]
and for the purpose of proving Theorem~\ref{bonn}
it will suffice to prove that the counting function $N_X\!\l(B\r)$ associated to this equation
satisfies
\begin{equation}
\label{sweeney}
N_X\!\l(B\r)\gg
B \l(\log B\r)^3.
\end{equation}

Our proof
makes use of the conic bundle structure present in $X.$ More specifically, $X$ is equipped with a dominant morphism
$\pi:X \to \P_\Q^1 $
such that
\[ \pi\l(x\r)=\begin{cases}
   [x_0:x_2] & \text{if} \ \l(x_0,x_2\r)\neq 0  \\
   [x_2^2+3 x_3^2:x_0^2+3x_1^2] & \text{if} \ \l(x_0^2+3x_1^2,x_2^2+3 x_3^2\r) \neq 0 .
  \end{cases}
\]
It can easily be verified
that the fibres $\pi^{-1}\l([s\!:\!t]\r)$
are the diagonal conics given by
\[Q_{s,t}:
\l(t^3-s^3\r)x^2
+\l(-3s\r)y^2
+\l(3t\r)z^2=0.
\]
whose discriminant equals
\[\Delta_{s,t}=-9\l(t^3-s^3\r)st.
\]
Define 
for any
$\l(s,t\r) \in 
\Zp^2$
the following norm on 
$\mathbb{R}^3,$
\[\|\!\l(x,y,z\r)\!\|_{s,t}:=\!
\max
\Big\{
|xs|,|xt|,|y|,|z|
\Big\}
\]
and the following counting function
\[
M_{s,t}\!\l(B\r):=
\#\l\{\b{x} \in \Zp^3, 
Q_{s,t}\l(\b{x}\r)=0,
\|\b{x}\|_{s,t}
\leq B\r\}.\]
It is evident that
for any
$B\geq 1$
we have
\begin{equation}
\label{equal}
N_X\!\l(B\r)=
\hspace{-0.5cm}
\Osum_{\substack{
|s|,|t| \leq B \\
Q_{\!s,t} \ \text{isotropic}\\
Q_{\!s,t} \ \text{non-singular} 
}}
\hspace{-0.5cm}
M_{s,t}\!\l(B\r),
\end{equation}
where 
$\Osum_{\hspace{-0.2cm}s,t}$ denotes summation
over $\l(s,t\r) \in \Z_{prim}^2.$
Indeed, for any
$\b{x}$ counted by 
$N_X\!\l(B\r)$
we let
$x\!:=\!\gcd(x_0,x_2),$
thus getting
$\l(s,t\r) \in \Zp^2$ with $x_0\!=\!sx$ and $\ x_2\!=\!tx.$
This shows that
$\l(x,y,z\r)$ is a primitive integer zero 
of the conic
$Q_{s,t}$
with $\|\!\l(x,y,z\r)\!\|_{s,t}\leq B.$
The fact that
zeros $\b{x}$ on the $3$ rational lines
\begin{equation*}
x_0=x_2=0, \ \
x_0-x_2=x_1-x_3=0, \ \
x_0-x_2=x_1+x_3=0.
\end{equation*}
of $X$ 
correspond to the zeros $\l(x,y,z\r)$
of singular
conics $Q_{s,t}$
follows
upon noticing
that 
$\b{x}$ lies on any of these lines
if and only if
$x_0 x_2\l(x_2-x_0\r)\!=\!0,$
which in turn is equivalent to 
the vanishing of the discriminant $\Delta_{s,t}.$

We estimate $M_{s,t}\!\l(B\r)$ via~\cite[Th.1]{sofos},
for 
any $\l(s,t\r)$ 
in the range
with $|s|,|t|\leq B^{\delta},$
where
$\delta \! \in \l(0,\frac{1}{40}\r)\!.$
In the notation of 
the theorem,
we have
$
\langle Q_{s,t}\rangle \! \ll \! B^{\frac{3}{20}}$
and
$K_{s,t} \!=\!2,$
thereby yielding
\[M_{s,t}\!\l(B\r)\!=\!
\frac{1}{2}
\sigma_\infty\l(s,t\r)
\prod_p\sigma_p\l(s,t\r)
B+O\!\l(B^{\frac{19}{20}}\r),
\]
where the
implied constant
is absolute 
and
the Hardy--Littlewood
densities 
$\sigma_\infty\l(s,t\r),\sigma_p\l(s,t\r)$
are defined 
in~\eqref{hl}
and~\eqref{realhl}
respectively.
To apply this to ~\eqref{equal},
notice that since we are interested in obtaining a lower bound for
$N_X\!\l(B\r),$ any extra conditions can be freely
imposed on the
summation over 
$s$ and $t.$
Letting
\[
\mathfrak{G}\!\l(x\r):=
\frac{1}{2}
\hspace{-0.5cm}
\Osum_{\substack{
|s|,|t|\leq x \\
Q_{\!s,t} \ \text{isotropic}\\
Q_{\!s,t} \ \text{non-singular} 
}}
\hspace{-0.5cm}
\sigma_\infty\l(s,t\r)
\prod_p\sigma_p\l(s,t\r)
\]
we
obtain
\[
N_X\!\l(B\r)\geq
B \ 
\mathfrak{G}\!\l(\!B^\frac{1}{40}\!\r)
+
O\!\l(B\r).
\]
Now~\eqref{sweeney}
implies that
Theorem~\ref{bonn}
would follow from
\begin{equation}
\label{marl}
\mathfrak{G}\!\l(x\r)\gg \l(\log x\r)^3,
\end{equation}
for all $x\geq 2.$ Establishing this estimate is the goal 
of \S\ref{s3}--\S\ref{s5}.
\section{Passing from cubic surfaces to divisor sums}
\label{s3}
Our aim in this section is to show that the quantity
$\mathfrak{G}\!\l(x\r)$ is
approximated by a divisor sum.
We will do so by finding explicit lower bounds for the Hardy--Littlewood densities
uniformly with respect to $s$ and $t.$
The following lemma will be used to facilitate the estimation of 
the $p$-adic densities.
\begin{lem}
\label{hl0}
Let $q\!\l(\b{x}\r)\!:=\!\sum_{i=1}^3a_i x_i^2$ be an
integral smooth
quadratic form
and let $p$ be an odd prime dividing exactly one coefficient, say
$a_1.$ Then for any $n\geq 1,$ the number of solutions
of $q\!\l(\b{x}\r)\!\equiv \! 0 \hspace{-0.2cm} \mod{p^n}$ such that
$p$ and
$x_3$ are coprime is at least
\[\l(1+\l(\frac{-a_2 a_3}{p}\r)\!\r)\!\l(1-\frac{1}{p}\r)p^{2n}
.\]
\end{lem}
\begin{proof}
It suffices to consider the case in which
$-a_2 a_3$ is a quadratic residue modulo $p$
since otherwise the statement is trivial.
Suppose therefore that there exists $t \in \Z/p\Z$
such that $a_2 t^2+a_3 \equiv 0\mod{p}.$ 
The proof is then completed via induction on $n$.

Letting $x_3$
run through the values $1,\ldots,p\!-\!1$ and $x_2\!=\!tx_3$
we deduce the validity of the statement for $n=1.$
Let us observe that each solution
$\b{x} \hspace{-0.12cm}\mod{\!p^n}$
of 
$q\l(\b{x}\r)\equiv 0 \hspace{-0.12cm}\mod{\!p^n}$
 with $p\nmid x_3$
satisfies the hypothesis of Hensel's lemma,
\[p\nmid \nabla\l(q\l(\b{x}\r)\r)=2
\l(a_1x_1,a_2x_2,a_3x_3\r)
,\]
and hence can be lifted to 
$p^2$ different
solutions 
$\b{y} \hspace{-0.12cm}\mod{\!p^{n+1}}$
of the equation
$q\!\l(\b{x}\r)\!\equiv \! 0 \hspace{-0.2cm} \mod{p^{n+1}}$
that must necessarily
satisfy
$p\nmid y_3.$ 
\end{proof}
Let us define
the arithmetic functions
\[
\div\l(n\r):=
\frac{\phi\l(n\r)}{n} 
2^{\omega\l(n\r)}\]
and
\[
\label{1}
\rr\l(n\r):=
\frac{\phi\l(n\r)}{n} 
\prod_{p|n}\Big(1+\chi_3\l(p\r)\Big),
\]
where
$\chi_3$
denotes
the
non--trivial character modulo $3,$
given by
\[\chi_3\l(n\r)=
  \begin{cases}
  0 & \text{if} \  n\equiv 0 \mod{3}  \\
  1 & \text{if} \  n\equiv 1 \mod{3}  \\
 -1 & \text{if} \  n\equiv 2 \mod{3} .
  \end{cases}
\]
For each $\l(s,t\r) \in \Z_{prim}^2$
and every prime $p$ 
we let
\[N_{s,t}^*\l(p^n\r):=\#\l\{\b{x} \mod{p^n}, 
Q_{s,t}\l(\b{x}\r)
\equiv 0 \mod{p^n}
,\  p\nmid\b{x}\r\}
\]
denote the number of primitive zeros $\hspace{-0.12cm} \mod{\!p^n}$
of 
the quadratic form $Q_{s,t}.$
It was shown in~\cite[Th.1]{sofos} that the following limit exists,
and 
its value
is referred to as
the Hardy--Littlewood $p$-adic
density, 
\begin{equation}
\label{hl}
\sigma_{p}\l(s,t\r)\!:=\!
\lim_{n\to \infty}N_{s,t}^*\l(p^n\r)p^{-2n}.
\end{equation}
The next lemma
reveals that
the evaluation 
of these
densities naturally gives birth to 
the function
$\rr$
which
plays the r\^{o}le of the detector
for the
conics $Q_{s,t}$
that are
isotropic over $\Q$.
\begin{lem}
\label{hl1}
Suppose that the integers
$s$
and $t$
satisfy 
$0\!<\!s\!<t,$
are coprime, and both equivalent to $1$ modulo $8.$
Then we have
\[
\prod_{p}
\sigma_p\!\l(s,t\r)
\gg
\div\!\l(t^3-s^3\r)
\rr\!\l(st\r)
,
\]
where the implied constant is absolute.
\end{lem}
\begin{proof}
We begin 
by calculating the 
$p$-adic densities for every prime $p\neq 2,3$
dividing the discriminant 
of the conic $Q_{s,t}.$
The coprimality of $s$ and $t$ ensures that
no two coefficients of the conic 
are divisible by $p$
and
hence we are allowed to
use Lemma~\ref{hl0}.

In the case that $p$ divides $s^3-t^3,$
we compute that 
\[\l(\!\frac{-\l(3t\r)\l(-3s\r)}{p}\!\r)=
\l(\frac{st}{p}\r)
=\l(\frac{s^3t}{p}\r)
=\l(\frac{t^4}{p}\r)=1,
\]
and hence we are provided with the estimate
$\sigma_p\l(s,t\r)\geq \div\!\l(p\r).$

The cases where $p|s$ or $p|t$ are symmetric and we therefore focus
on the latter. A similar computation yields
\[\l(\!\frac{-\l(t^3-s^3\r)\l(-3s\r)}{p}\!\r)
=
\l(\frac{-3s^4}{p}\r)
=
\l(\frac{-3}{p}\r)
.\]
Alluding to quadratic reciprocity reveals that
the value of 
the Legendre symbol $\!\l(\frac{-3}{p}\r)\!$
equals $\chi_3\l(p\r)\!,$
thereby yielding the estimate
$\sigma_p\!\l(s,t\r)\!\geq \! \rr\!\l(p\r).$

We recall at this point that if the conic $Q_{s,t}$ has zeros 
over $\Q_p$ 
then the $p$-adic density $\sigma_p\l(s,t\r)$ does not
vanish,
in which case
the estimate \[\sigma_p\l(s,t\r)
\geq
1-\frac{1}{p^2}
\]
follows from 
~\cite[Eq.(5.2)]{sofos}
or via Hensel lifting. Therefore the validity 
of the lemma
would follow upon 
showing that
the conic $Q_{s,t}$
is isotropic over $\Q_p,$
for all primes $p$
not considered so far.
 
If the conic is nonsingular over 
$\Q_p,$
i.e. $p\nmid \Delta_{s,t},$
then it is a well--known fact that
it is isotropic over $\Q_p.$
We are thus left
with
considering the cases $p=2$ and
$3.$
Hilbert's product formula,
implies that
it suffices to consider
solubility merely over $\Q_2.$
Recalling that $1+8\Z_2 \subseteq \Q_2^{\times 2}$
shows that $9st=u^2$
for some $u \in \Q^{\times}_2.$
We observe that 
$\l(0,u,3s\r)$
is a zero of $Q_{s,t}$ over $\Q_2,$
which concludes the proof of the lemma.
\end{proof}
We next turn our attention to the evaluation 
of the  archimedean
Hardy--Littlewood density. 
It is 
defined as 
\begin{equation}
\label{realhl}
\sigma_\infty\l(s,t\r):=
\lim_{\epsilon \to 0}
\frac{1}{2\epsilon}
\int_{\substack{
\\|Q_{s,t}\l(\b{x}\r)|\leq \epsilon
\\
\|\b{x}\|_{s,t}\leq 1
}}
1 \ 
\mathrm{d}\b{x}.\end{equation}
Let
$\B:=
\l(\frac{1}{4},\frac{1}{2}\r]
\times
\l(\frac{1}{2},1\r]
$
be an interval of $\R^2.$
\begin{lem}
\label{hl2}
Suppose that $\l(s,t\r) \in \Z_{prim}^2$ belongs to $\B 2^n$
for some natural number $n\geq 2.$
We have the estimate
\[\sigma_\infty\l(s,t\r)
\gg
4^{-n}
,
\]
with an absolute implied constant.
\end{lem}
\begin{proof}
Assume that
$\epsilon \in \l(0,\frac{1}{2}\r)$
throughout the proof
and notice that the condition
$\l(s,t\r) \in \B 2^n$
implies that
$0<
\!
\frac{t}{4}
\!
<
\!
s
\!
<\frac{t}{2}.$
It is clear that
whenever 
$y\in [0,\frac{1}{4}]$
and
$(x,z)$ lies in the interior of the region 
$R\l(y,\epsilon\r)$
defined by
\[\l(3s\r)y^2-\epsilon
\leq
\l(t^3-s^3\r)x^2
+\l(3t\r)z^2
\leq
\l(3s\r)y^2
+\epsilon
\]
then
$tx$
and $z$
are both bounded in modulus by $1,$
thus leading to
\[
\int_{\substack{
\\|Q_{s,t}\l(\b{x}\r)|\leq \epsilon
\\
\|\b{x}\|_{s,t}\leq 1
}}
1
\mathrm{d}\b{x}
\geq
\int_0^{\frac{1}{4}}
\text{vol}\l(R\l(y,\epsilon\r)\r)
\mathrm{d}y.
\]
Noting that
for $a,b>0$
the area of the ellipse 
$ax^2+by^2=1$ 
is equal to
$\pi(ab)^{-\frac{1}{2}},$
we find that the volume of $R\l(y,\epsilon\r)$
is at least
$\epsilon t^{-2}.$
We therefore deduce that
\[
\frac{1}{2\epsilon}
\int_{\substack{
\\|Q_{s,t}\l(\b{x}\r)|\leq \epsilon
\\
\|\b{x}\|_{s,t}\leq 1
}}
1
\mathrm{d}\b{x}
\geq
\frac{1}{8t^2}
\]
which leads to the desired result.
\end{proof}
Recall that 
$\Osum_{\hspace{-0.2cm}s,t}$
denotes summation
over coprime integers $s$ and $t$
and define the
sum
\[
D\!\l(x\r):=
\hspace{-0.5cm}
\Osum_{\substack{\l(s,t\r) \in \B x\\
 s,t   \equiv 1 \hspace{-0.12cm}\mod{\!8}}}
\hspace{-0.3cm}
 \div\!\l(t^3-s^3\r)  
\rr\l(st\r)
\]
for any $x\geq 2.$
The following result is obvious.
\begin{lem}
\label{theem}
We have the estimate
\[\mathfrak{G}\!\l(x\r)\gg 
\hspace{-0.5cm}
\sum_{\substack{ n \in \mathbb{N} \\ n \in (1,\log_2 x]}}
\hspace{-0.3cm}
4^{-n} \
D\!\l(2^n\r),
\]
for all
$x\geq 4$
with an absolute implied constant.
\end{lem}
In light of~\eqref{marl},
we are led to the conclusion that
Theorem~\ref{bonn}
would follow upon proving the estimate
\begin{equation}
\label{marl1}
D\!\l(x\r)
\gg
x^2\l(\log x\r)^3,
\end{equation}
for all $x\geq 4.$
Notice that this is 
a lower bound of the correct order of magnitude.
\section{Estimation of divisor sums}
\label{s4}
We gather level of distribution 
results regarding functions related to $\rr\l(n\r)$
before evaluating $D\!\l(x\r)$
in~\S\ref{s5}.
We denote the summation
over  residue classes $a \! \in \! [0,q)$
which are coprime to $q$
by 
$\Osum_{\hspace{-0.2cm}a\hspace{-0.12cm} \mod{\!q}}$
and the inverse of $a\hspace{-0.12cm} \mod{\!q}$
by $\bar{a}.$
The Kloosterman sums $S\l(a,b;c\r)$
are defined by
\[S\!\l(a,b;c\r)=
\Osum_{x \hspace{-0.12cm}\mod{\!c}}\!
e\!\l(\frac{ax+b\bar{x}}{c}\r),
\]
for any integers $a,b,c,$
where we use the notation
$e\!\l(z\r)\!:=e^{2 \pi i z}.$ 
The Weil bound~\cite[Eq.(1.60)]{iwa} states that
\begin{equation}
\label{weil}
|S\!\l(a,b;c\r)|\leq
\tau\!\l(c\r)\gcd(a,b,c)^{\frac{1}{2}}
|c|^{\frac{1}{2}}.
\end{equation}
\begin{lem}
\label{poisson1}
For any integers $a,q$ with $q$ coprime to $3a$ 
and any $X\geq 1,$
we have
\[\sum_{\substack{
n\leq X
\\
n\equiv a 
\hspace{-0.12cm}
\mod{\!q}}}
\hspace{-0.5cm}
\l(1\!\ast\!\chi_3\r)\!\l(n\r)
=
\frac{\pi}{3^{\frac{3}{2}}}
\tilde{c}\!\l(q\r)
\frac{X}{q}
+O\l(X^{\frac{1}{3}} 
q^{\frac{1}{2}} 
\tau^2\!\l(q\r)\r),
\] 
where 
$\ast$ denotes the Dirichlet convolution,
the implied constant is absolute
and
\[
\tilde{c}\!\l(q\r)
:=\sum_{d|q}
\frac{\chi_3\!\l(d\r)}{d}
\mu\!\l(d\r).
\]
\end{lem}
\begin{proof}We use the function 
$g$ defined in the proof of~\cite[Cor.4.9]{iwa}.
Introducing additive characters to detect the congruence
$n\equiv a \hspace{-0.12cm} \mod{q},$
we arrive at
the following upper bound for the sum
in the statement of the lemma,
\[\sum_{\substack{
n\leq X
\\
n\equiv a 
\hspace{-0.12cm}
\mod{\!q}}}
\hspace{-0.5cm}
\l(1\!\ast\!\chi_3\r)\!\l(n\r)
g\!\l(n\r)=
\frac{1}{q}
\sum_{m
\hspace{-0.12cm}
\mod{\!q}}
\hspace{-0.3cm}
e\!\l(\!-\frac{ma}{q}\r)
\sum_{n=1}^{\infty}
\l(1\ast\chi_3\r)\!\l(n\r)
e\!\l(\frac{mn}{q}\r)
g\!\l(n\r)
.
\]
We partition
the summation over
$m$ 
according
to the value of the 
$\gcd(m,q)\!=\!\frac{q}{d}$ which leads to
\[
\frac{1}{q}
\sum_{d|q}
\
\Osum_{k
\hspace{-0.12cm}
\mod{\!d}}
\hspace{-0.1cm}
e\!\l(\!-\frac{ka}{d}\r)
\sum_{n=1}^{\infty}
\l(1\!\ast\!\chi_3\r)\!\l(n\r)
e\!\l(\frac{kn}{d}\r)
g\!\l(n\r)
.\]
We will use~\cite[Eq.(4.70)]{iwa}
to estimate the inner sum. Notice that
although the function
$g\l(x\r)$ is not smooth, 
the statement is still valid.
Recall that the values 
of the $L$-function 
and the Gauss sum
corresponding to the character $\chi_3$
are $\frac{\pi}{3^{\frac{3}{2}}}$
and 
$i3^{\frac{1}{2}}$
respectively.
This yields
\begin{align*}
\sum_{n=1}^{\infty}
\l(1\!\ast\!\chi_3\r)\!\l(n\r)
&e\!\l(\frac{kn}{d}\r)
g\!\l(n\r)=
\frac{\pi}{3^{\frac{3}{2}}}
\frac{\chi_3\l(d\r)}{d}
\int_0^{\infty}
\hspace{-0.2cm}
g\!\l(x\r)\mathrm{d}x
\\
&-
\frac{2\pi i}{3^{\frac{1}{2}}}
\frac{\chi_3\l(d\r)}{d}
\sum_{n=1}^{\infty}
\l(1\!\ast\!\chi_3\r)\!\l(n\r)
e\!\l(\frac{\overline{(3k)}n}{d}\r)
h\!\l(\frac{4n}{3d^2}\r),
\end{align*}
where $h$ 
is defined in~\cite[Eq.(4.36)]{iwa}
and
denotes the Hankel type transform of $g.$
Interchanging the 
order of summation
gives rise to Kloosterman 
sums,
\begin{align*}
\Osum_{k
\hspace{-0.12cm}
\mod{\!d}}
\hspace{-0.3cm}
e\!\l(\!-\frac{ka}{d}\r)
&\sum_{n=1}^{\infty}
\l(1\!\ast\!\chi_3\r)\!\l(n\r)
h\!\l(\frac{4n}{3d^2}\r)
e\!\l(\frac{\overline{(3k)}n}{d}\r)
\\
=&\sum_{n=1}^{\infty}
\l(1\!\ast\!\chi_3\r)\!\l(n\r)
h\!\l(\frac{4n}{3d^2}\r)
S\!\l(-a,\bar{3}n;d\r).
\end{align*}
The
bound~\cite[Eq.(4.44)]{iwa}
when combined with~\eqref{weil}
proves that the last expression
is $\ll
\!
\tau\!\l(d\r)
d^{\frac{5}{2}}
(X/Y)^{\frac{1}{2}}
.$
Putting everything together
yields
\begin{align*}
\sum_{\substack{
n\leq X
\\
n\equiv a 
\hspace{-0.12cm}
\mod{\!q}}}
\hspace{-0.5cm}
\l(1\!\ast\!\chi_3\r)\!\l(n\r)
g\!\l(n\r)
&=
\frac{\pi}{3^{\frac{3}{2}}}
\frac{1}{q}
\l(
\int_0^{\infty}
\hspace{-0.2cm}
g\!\l(x\r)\mathrm{d}x\r)
\sum_{d|q}
\frac{\chi_3\!\l(d\r)}{d}
\Osum_{k
\hspace{-0.12cm}
\mod{\!d}}
\hspace{-0.3cm}
e\!\l(\!-\frac{ka}{d}\r)
\\
&+O
\Big(
(X/Y)^{\frac{1}{2}}
q^{\frac{1}{2}}
\tau^2\!\l(q\r)
\Big).
\end{align*}
This with 
$\Osum_{
\hspace{-0.2cm}
k
\hspace{-0.12cm}
\mod{\!d}}
\hspace{-0.05cm}
e\!\l(\!-\frac{ka}{d}\r)\!=\!\mu\l(d\r)$
and
$\int_0^{\infty}
g\l(x\r)\mathrm{d}x
\!=\!X+\frac{Y+1}{2}
$
proves the one-side estimate required for this lemma, on taking
$Y\!=\!X^{\frac{1}{3}}.$ The desired lower bound is obtained in an
identical way by
using the weight function
$g\!\l(x\r)\!=\!
\max\{x,1,\l(X-x\r)Y^{-1}\}$
instead.
\end{proof}
We will later need to 
estimate averages over arithmetic progressions
of functions belonging to a rather large class of multiplicative functions.
Our results are best formulated in terms 
of the group of functions
\[G\!:=\!\l\{
f\!:\!\mathbb{N} \to \mathbb{R}_{\geq 0}, f \ 
\text{multiplicative},f\!\l(p\r)\!=\!1\!+\!O_{\!f\!}\l(\frac{1}{p}\r) 
\text{for all primes} \ p\r
\}.
\] 
Note that
$G$
is an abelian group
with respect to
pointwise multiplication
and
contains elements
such as $\phi\!\l(q\r)\!/q$
and
the function $\tilde{c}\!\l(q\r)$ defined in 
Lemma~\ref{poisson1}. The bound~\eqref{eq:1}
reveals that each element $f \in G$ satisfies 
$\mu^2\!\l(n\r)\!f\!\l(n\r)\!\ll_{\epsilon,f}\!n^{\epsilon}$
for any $\epsilon\!>\!0.$
\begin{lem}
\label{tors}Let $f$ be a positive
function
such that
either
$f \! \in \!  G$
or
$$f\!\l(n\r)=g\!\l(n\r)\prod_{p|n}\l(1+\chi_3\!\l(n\r)\r)$$
for some
$g \!\in \! G.$
Then
there exists
a function
$\widehat{f}\! \in\! G$
and a positive constant $c\!\ll_{\!f\!}\!1,$
both of which depend on $f,$
such that
for any integers $q,\!a,\!k$ 
with $q$ coprime to $3ak$
and $x \!\geq\! 1,\epsilon\!>\!0$
we have
\[\sum_{\substack{n\leq x \\
n\equiv a \hspace{-0.12cm}\mod{\!q}
\\
\gcd(n,k)=1
\\
n \ \text{squarefree}}}
\hspace{-0.5cm}
f\!\l(n\r)=
c 
\frac{\phi\!\l(k\r)}{k}
\widehat{f}\!\l(kq\r)
\frac{x}{q}+
O_{\epsilon,f}\!\l(x^{\frac{1}{2}+\epsilon}q^{\frac{1}{2}}\tau^2\!\l(kq\r)\r).
\]

\end{lem}
\begin{proof}
The fact that $g \! \in \! G$ implies that 
in the second case we have
\[f\!\l(p\r)\!=\!1\!+\chi_3\!\l(p\r)+\!O_{\!f\!}\l(\frac{1}{p}\r)
\]
for all primes $p.$
Hence there exists
$\delta \in \{0,1\}$
such that
the quantity
\[M\!:=1+\sup_{p}\Big|p\l(f\l(p\r)-1-
\delta\chi_3\l(p\r)\r)\Big|
\]
is well--defined.
Define the function
$\theta$ by letting
$\theta\l(n\r)\!=\!f\!\l(n\r)\mu^2\!\l(n\r)$
for $n$ coprime to $k$ and  
$\theta\!\l(n\r)\!=\!0$
otherwise, therefore turning
the sum in the lemma
into \[\sum_{\substack{n\leq x \\ n \equiv a \hspace{-0.12cm}  \mod{\!q}}}
\!
\theta\!\l(n\r).\]
We define
$\theta'\!\l(n\r)\!:=\theta \ast \mu$
if $\delta=0$
and
$\theta'\l(n\r):=\theta\ast \mu \ast \mu\chi_3$
if $\delta=1.$
We can then verify 
by induction
that regardless of the value
of $\delta$
we have 
\[
|\theta'\!\l(p^k\r)| \leq
  \begin{cases}
   \frac{M}{p} & \text{if} \  k=1  \ \text{and} \ p\nmid a  \\
    3+2M & \text{if} \ k=2,3 \ \text{and} \ p\nmid a  \\
     2 &  \text{if} \ k=1,2 \  \text{and}\ p|a \\
     0 &  \text{if} \ k=3 \  \text{and}\ p|a \\
    0 & \text{if} \ k\geq 4.
  \end{cases}
\]
Hence for each $\epsilon>0$
we have the estimate
\begin{align*}
\sum_{n\leq x}
\frac{|\theta'\l(n\r)|}{n^{\frac{1}{2}+\epsilon}}
&\ll
\prod_{p\leq x}\l(1+\frac{3+3M}{p^{\frac{3}{2}+\epsilon}}+\frac{3+2M}{p^{1+2\epsilon}}\r)
\prod_{p|a}\l(1+\frac{4}{p^{\frac{1}{2}+\epsilon}}\r)
\\
&\ll_{\epsilon,M}
\tau\!\l(a\r)
.\end{align*}
Therefore the case $\delta\!=\!0$ of the lemma follows by~\cite[Lem.2]{bret} with $k=\frac{1}{2}+\epsilon.$
Hence let $\delta=1.$ 
Since
$\chi_3\mu$ is the inverse of $\chi_3,$
we deduce that
$\theta=\theta'\ast\l(1\ast\chi_3\r),$
which
leads to
\[\sum_{\substack{n\leq x \\ n \equiv a \hspace{-0.12cm}  \mod{\!q}}}
\hspace{-0.2cm}
\theta\!\l(n\r)
=\hspace{-0.35cm}
\sum_{\substack{n\leq x \\ \gcd(n,q)=1}}
\theta'\!\l(n\r)
\hspace{-0.2cm}
\sum_{\substack{m\leq x/n \\ nm \equiv a \hspace{-0.12cm}  \mod{\!q}}}\!
\hspace{-0.2cm}
\l(1\!\ast\!\chi_3\r)\!\l(m\r).
\]
An application of Lemma~\ref{poisson1}
yields
\[\frac{\pi}{3^{\frac{3}{2}}}
\frac{x}{q}
\tilde{c}\!\l(q\r)
\l(\!
\sum_{\substack{n\leq x \\ \  \gcd(n,q)=1}}\!
\frac{\theta'\l(n\r)}{n}\r)
+O_{\epsilon,M}\l(
x^{\frac{1}{2}+\epsilon}
q^{\frac{1}{2}}
\tau^2\!\l(aq\r)\r)
\]
valid for any $\epsilon\!>\!0.$
The sum over $n$ is absolutely convergent
and extending the summation to infinity
introduces a negligible error term.
Factoring the series into Euler products
reveals that
the statement of the lemma
holds
with the constant
\[
c:=\frac{\pi}{3^{\frac{3}{2}}}
\prod_{p}
\l(1+\frac{f\l(p\r)}{p}\r)
\l(1-\frac{1}{p}\r)
\l(1-\frac{\chi_3\l(p\r)}{p}\r)
\]
and the function
\[
\widehat{f}\l(n\r):=\tilde{c}\!\l(n\r)
\prod_{p|n}
\l(1+\frac{f\l(p\r)}{p}\r)^{-1}
\!\l(1-\frac{1}{p}\r)^{-1}
\!\l(1-\frac{\chi_3\l(p\r)}{p}\r)^{-1}\!.
\]
\end{proof}
We will later consider divisor sums
involving binary forms for which a two dimensional 
version of the previous lemma shall be required.
\begin{lem}
\label{coprim}
Let $f_1$ and $f_2$
be functions satisfying the assumption
of Lemma~\ref{tors}. Then
there exists
a  function
$\widehat{f} \in G$
and
a positive constant $c$,  both of which depend 
on $f_1$ and $f_2,$
such that
for any
$x,y\!\geq \!1, \epsilon \!>\!0$
and integers
$q,a,k,\sigma,\tau$ with
$q$ coprime to $3ak\sigma\tau$
we have
\[\!
\Osum_{\substack{
s \leq y,
t \leq x \\
\gcd(st,k)=1 \\
s,t \ \text{squarefree} \\
\l(s,t\r)\equiv \l(\sigma,\tau\r)\hspace{-0.12cm} \mod{\!q}
}}
\hspace{-0.5cm}
f_1\!\l(s\r)
f_2\!\l(t\r)
=
c \ 
\widehat{f}\!\l(kq\r)
\!\l(\!\frac{\phi\!\l(k\r)}{k}\!\r)^2
\!\frac{xy}{q^2}+O_{\epsilon,f_1,f_2}\!
\l(
\frac{xy}{{\min(x,y)}^{\frac{1}{2}-\epsilon}}
 \ q^{3}k^\epsilon\r).
\]
\end{lem}
\begin{proof}
As in the proof of the previous lemma,
there exist $\delta_1, \delta_2 \in \{0,1\}$
such that the following quantity is well--defined
\[M\!:=1+\max_{i=1,2} \ \sup_{p}\Big|p\l(f\l(p\r)-1-
\delta_i\chi_3\l(p\r)\r)\Big|.\]
We may assume $y\leq x$ without loss of generality.
Using 
M\"{o}bius inversion
to deal with the coprimality of $s$ and $t$
implies that the sum appearing in the lemma equals
\[\sum_{\substack{
m \leq  y/2 \\
\gcd(m,kq)=1
}}
\hspace{-0.5cm}
\mu\!\l(m\r)
f_1\!\l(m\r)
f_2\!\l(m\r)
\hspace{-0.3cm}
\sum_{\substack{
s \leq y/m \\
\gcd(s,km)=1 \\
s \ \text{squarefree} \\
s m\equiv \sigma \hspace{-0.12cm} \mod{\!q}
}}
\hspace{-0.5cm}
f_1\!\l(s\r)
\sum_{\substack{
t \leq x/m \\
\gcd(t,km)=1 \\
t \ \text{squarefree} \\
t m\equiv \tau \hspace{-0.12cm} \mod{\!q}
}}
\hspace{-0.5cm}
f_2\!\l(t\r).
\]
Let us observe that the previous lemma supplies the estimates
\begin{equation}
\label{bnd}
\mu^2\!\l(n\r)\!f_i\!\l(n\r),\mu^2\!\l(n\r)\!\widehat{f}_{i}\!\l(n\r)
\ll_{\epsilon,M}\!n^{\epsilon},
|c_{i}|\!
\ll_{M}\!1 \ \text{for} \ i=1,2,
\end{equation}
valid for any $\epsilon\!>\!0,$ due to the definition of
$G$
and $\eqref{eq:1}.$
We are
therefore
provided with an error term
bounded by $\ll_{\epsilon,M}
xy^{\frac{1}{2}+\epsilon}
q^{2+\epsilon}k^{\epsilon}$
and with the main term
\begin{align*}
&c_{1}c_{2}
\
\widehat{f_1}\!\l(kq\r)
\widehat{f_2}\!\l(kq\r)
\l(\!\frac{\phi\!\l(k\r)}{k}\!\r)^2
\!\frac{xy}{q^2}
\ \times \\
&\times \!
\sum_{\substack{
m \leq  y/2 \\
\gcd(m,kq)=1
}}
\hspace{-0.5cm}
\frac{\mu\l(m\r)}{m^2}
{f_1}\!\l(m\r)
{f_2}\!\l(m\r)
\widehat{f_1}\!\l(m\r)
\widehat{f_2}\!\l(m\r)
\l(\!\frac{\phi\!\l(m\r)}{m}\!\r)^2
.
\end{align*}
The bounds~\eqref{bnd}
enable us to extend the summation over $m$
to infinity, 
introducing a negligible error.
Comparing the ensuing Euler factors
concludes the proof of the lemma.
\end{proof}

\section{Proof of Theorem~\ref{bonn}}
\label{s5}
We dedicate this section to the proof 
of~\eqref{marl1}.
Denote the summation
over  residue classes $\sigma,\tau \! \in \! [0,d)$
for which $\gcd(\sigma,\tau,d)=1$
by 
$\Osum_{\hspace{-0.2cm}\l(\sigma,\tau\r)\hspace{-0.12cm} \mod{\!d}}.$
Defining the function
$\u\!\l(n\r)\!
:=
\!\mu^2\!\l(n\r)
\prod_{p|n}\!\l(1-\frac{2}{p}\r)$
allows us to write
$\div\!\l(n\r)\!=\!\sum_{d|n} \widetilde{1}\!\l(d\r).$
We insert this into $D\!\l(x\r)$
and 
invert the order of summation. The
non--negativity 
of the values assumed by 
$\u$
enables us to restrict the summation,
arriving at
\begin{equation}
\label{fugue}
D\!\l(x\r)
\geq
\hspace{-0.3cm}
\Osum_{\substack{d_1,d_2\leq x^{\epsilon}
\\ \gcd(d_1d_2,6)=1
}}
\hspace{-0.2cm}
\u\l(d_1d_2\r)
\Osum_{\substack{(\sigma,
\tau)
\hspace{-0.12cm} \mod{\!8 d_1 d_2}
\\
\sigma,\tau\equiv 1 \hspace{-0.12cm}\mod{\!8}
\\
d_1|\sigma^2
+\sigma\tau  +
\tau^2 
\\
d_2|\sigma-\tau
\\
}}
\Osum_{\substack{
\l(s,t\r) \in \B x
\\
s,t \ \text{squarefree}
\\
\l(s,t\r)\equiv
\l(\sigma,\tau\r)
\hspace{-0.12cm}
\mod{\!8d_1d_2}
}}
\hspace{-0.7cm}
\rr\l(st\r),
\end{equation}
valid for any $0\!<\!\epsilon\!<\!1.$
We use Lemma~\ref{coprim}
with
$f_1=f_2=\rr$
to estimate the sum over $s$ and $t.$
We are allowed to do so
since $\phi\l(n\r)\!/ n$ is an element of the group $G$
and therefore $\rr$ satisfies the hypotheses of 
Lemma~\ref{tors}.  
We are thus led to
\[\Osum_{\substack{
\l(s,t\r) \in \B x
\\
s,t \ \text{squarefree}
\\
\l(s,t\r)\equiv
\l(\sigma,\tau\r)
\hspace{-0.12cm}
\mod{\!8d_1d_2}
}}
\hspace{-0.7cm}
\rr\l(st\r)=c_0 \
f_0\l(d_1d_2\r)
\frac{x^2}{d_1^2d_2^2}
+O_{\epsilon}\l(x^{\frac{7}{4}+\epsilon}d_1^3d_2^3\r)
\]
for any $\epsilon\!>\!0,$
where $f_0 \in G$
and $c_0$ is an absolute positive constant.
Taking into account that
the number of
solutions of
$x^2+x+1
\!\equiv\! 0
\hspace{-0.12cm}\mod{\!p}$
equals $1\!+\!\l(\frac{-3}{p}\r)$
for any prime $p$
and inserting
the previous estimate
into~\eqref{fugue}
leads to
\begin{equation}
\label{theme}
D\l(x\r)
\geq
c_1 x^2
\hspace{-0.2cm}
\Osum_{\substack{d_1,d_2\leq x^{\epsilon}
\\ \gcd(d_1d_2,6)=1
}}
\hspace{-0.2cm}
\frac{\u\!\l(d_1d_2\r)}{d_1 d_2}
f_0\!\l(d_1d_2\r)
\prod_{p|d_1}
\l(1+\chi_3\!\l(p\r)\r)
+O_\epsilon\l(x^{\frac{7}{4}+10\epsilon}\r),
\end{equation}
where $c_1$ is an absolute positive constant.
We observe that the functions 
$\u$ and $f_0$
are elements of the group $G$
and so does their product.
Hence  the functions
\[f_1\!\l(n\r)=
\u\!\l(n\r)
f_0\!\l(n\r)
\prod_{p|n}
\l(1+\chi_3\!\l(n\r)\r)
, 
\quad
f_2\!\l(n\r)=\u\!\l(n\r)
f_0\!\l(n\r),
\]
fulfill the hypotheses of
Lemma~\ref{coprim}.
Therefore for any $\epsilon\!>\!0$
we obtain
\[
\Osum_{\substack{
y/2<d_1\leq y
\\
x/2<d_2 \leq  x
\\
\gcd(d_1 d_2,6)=1}}
\hspace{-0.2cm}
\frac{\u\!\l(d_1d_2\r)}{d_1 d_2}
f_0\!\l(d_1d_2\r)
\prod_{p|d_1}
\l(1+\chi_3\!\l(p\r)\r)
=
c_2
xy
+O_{\epsilon}\l(\frac{xy}{\min^{\frac{1}{4}-\epsilon}(x,y)}\r)
\]
with an absolute positive constant $c_2.$
Partitioning
in dyadic intervals
shows that
the sum appearing in~\eqref{theme} is larger than
\[\sum_{\ \ \ 1\leq i \leq j \leq \epsilon \log_2 x}
\hspace{-0.5cm}
2^{-i-j}
\hspace{-0.3cm}
\Osum_{\substack
{2^{i-1}<d_1\leq 2^{i}
\\
2^{j-1}<d_2\leq  2^{j}
\\
\gcd(d_1 d_2,6)=1
}}
\hspace{-0.4cm}\u\!\l(d_1d_2\r)
\frac{\u\!\l(d_1d_2\r)}{d_1 d_2}
f_0\!\l(d_1d_2\r)
\prod_{p|d_1}
\l(1+\chi_3\!\l(p\r)\r).
\]
Using the previous estimate for
each inner sum
proves~\eqref{marl1}
from which Theorem~\ref{bonn} follows.

\section{Proof of Theorem~\ref{batt}}
\label{s6}
Recall the definition of $Y$ in~\eqref{btt}.
The anticanonical divisor on $Y$ 
is $-K_Y=\c{O}(3,1)$
which provides the height $H$
defined as follows.
For a point 
$\l(x,y\r) \in \P_{\Q}^3\times \P_{\Q}^3,$
we choose $\b{x}, \b{y} \in \Zp^4,$
unique up to sign, so that
$
\l(x,y\r)
=
\l([\b{x}],
[\b{y}]
\r)
$ and we let
\[H\l(x,y\r) := 
|\b{x}|^3
|\b{y}|
,\]
where $|\cdot|$
denotes the usual supremum norm in $\R^4.$
The counting function, defined in~\eqref{nam0},
takes the following shape.
For any Zariski open subset $U$ of $Y$
we set
\begin{equation}
\label{lms}
N\!\l(U,B\r)=\frac{1}{4}
\#\l\{\b{x},\b{y} \in \Zp^4,
\l([\b{x}],
[\b{y}]\r) \in U,
|\b{x}|^3
|\b{y}|\leq B
\r\}.
\end{equation}
Define the map $\pp:Y \to \P^3$
by $\pp\l(\b{x},\b{y}\r)=\b{x}.$
The image of $U$
under $\pp$
forms a Zariski open set
and it therefore
intersects
the Zariski dense subset of $\P_\Q^3$
given by
\[\l\{[t_0^3:\ldots:t_3^3]:t_0,\ldots,t_3 \in \Q^*
\r\}.\]
Letting $Y_\b{t}$ stand for the 
corresponding cubic surface
\[Y_\b{t}: \ \ \sum_{i=0}^3
t_i^3
y_i^3=0,
\]
we are provided with some $\b{t}=\l(t_0,\ldots,t_3\r) \in \Z_{prim}^4,$
with $\prod_{i=0}^3 t_i\neq 0,$ such that
$U_{\b{t}}:=U \cap Y_\b{t} $ is non--empty.
We fix the choice of $\b{t}$ 
for the rest of this section.
The surface $Y_\b{t}$ is irreducible
since $\prod_{i=0}^3 t_i\neq 0,$ 
which implies that
the closed subvariety
$Y_\b{t}\setminus U_\b{t}$ is a finite union
$\cup_{i=1}^r C_i$
of curves or points in $\P^3.$
Letting
\[N\!\l(U_\b{t},B\r)=\#\l\{\b{y} \in U_\b{t}\l(\Q\r), H\l(\b{y}\r)\leq 
B
\r\}
\]
we deduce that
\[
N\!\l(U,B\r)
\gg
N\!\l(U_\b{t},\frac{B}{\ |\b{t}|^9}\r)
\]
for all
$B\geq |\b{t}|^9.$
It therefore
suffices for the purpose of establishing Theorem~\ref{batt}
to prove
the lower bound
\begin{equation}
\label{batnk}
N\!\l(U_\b{t},B\r) \gg B \l(\log B\r)^3
\end{equation} for all $B \geq 3.$
We can therefore assume without loss of generality 
that
$U_\b{t}$ contains none of the lines of $Y_\b{t}.$
Letting
$Y'_\b{t}$ denote the Zariski open subset of $Y_\b{t}$
where the lines have been excluded,
we deduce that
\[
N\!\l(U_\b{t},B\r)
\geq
N\!\l(Y'_\b{t},B\r)
-\sum_{\substack{ 1\leq i \leq r \\ 
C_i
\neq \text{line}}}
N\!\l(C_i,B\r).
\]
Hence if $C_i$ is a curve, its degree will be at least $2$
in which case Theorem 1.1 in \cite{walsh} shows that $N\!\l(C_i,B\r)\ll_{C_i} B.$
Noting that this estimate
trivially holds if $C_i$ is a point,
we have shown that
\begin{equation}
\label{inclusion}
N\!\l(U_\b{t},B\r)
\geq
N\!\l(Y'_\b{t},B\r)
+O\!\l(B\r),
\end{equation}
where the implied constant depends at most on $r,\b{t}$ and $C_i$.
\begin{lem}
\label{fer}
We have 
\[N\!\l(Y'_\b{t},B\r) \gg_{\b{t}}
B \l(\log B\r)^3\]
for all $B \geq 3.$
\end{lem}
\begin{proof}
Let $T:=|t_0t_1t_2t_3|\neq 0$
and recall that $N\!\l(B\r)$ 
denotes the counting function
associated to the Fermat cubic surface.
Alluding to Theorem~\ref{bonn}
implies that
it suffices to prove
that
\begin{equation}
\label{one}
N\!\l(Y'_\b{t},B\r) 
\geq
N\!\l(\frac{B}{T}\r)
\end{equation}
for all $B\geq T.$
For any
$\b{z} \in \Zp^4$ 
counted by
$N\!\l(B\r)$ define $[\b{y}] \in \mathbb{P}_{\Q}^3$
via $y_i:=\frac{z_i}{t_i}$
and notice that $\b{y}$ lies on $Y'_\b{t}.$
We have $$[\b{y}]=
[z_0t_1t_2t_3:\ldots:z_3t_0t_1t_2]$$
and we observe that for $d:=\gcd(z_0t_1t_2t_3,\ldots,z_3t_0t_1t_2),$
we get
\[H\!\left(\b{y}\right)=\frac{\max\l\{|z_0t_1t_2t_3|,\ldots,|z_3t_0t_1t_2|\r\}}{d}
\leq BT,
\]
which proves
\eqref{one}.
\end{proof}
Combining the estimate
~\eqref{inclusion}
and Lemma~\ref{fer}
proves~\eqref{batnk}
from which the validity of
Theorem~\ref{batt} is inferred.

\bibliographystyle{amsalpha}
\bibliography{fermat}
\end{document}